\documentclass[11pt]{article}
\usepackage{amsmath,amssymb,latexsym,float,epsfig,subfigure}
\usepackage{latexsym, natbib}
\usepackage{amsthm}
\usepackage{epsfig}
\usepackage{algpseudocode}
\usepackage{subfigure}
\usepackage{scalefnt}
\usepackage{booktabs}
\usepackage[english]{babel}
\usepackage{graphicx}
\usepackage{titlesec}
\usepackage{epsfig}
\usepackage{enumerate}
\usepackage{caption}
\usepackage{subfigure}
\usepackage{booktabs}
\usepackage[english]{babel}
\usepackage{graphicx}
\usepackage{multirow}
\usepackage{rotating}
\usepackage{fancyhdr}
\usepackage{setspace}
\usepackage[section]{placeins}
\usepackage{breqn}

\usepackage{pseudocode}
\usepackage{algorithmicx}
\usepackage{algorithm}
\usepackage{url} 
\usepackage{appendix}

\topmargin by -\headsep \textheight 8.9in \oddsidemargin 0pt
\evensidemargin \oddsidemargin \marginparwidth 0.5in \textwidth
6.5in
\begin{document}
\title{Capacity Allocation in Queuing Systems with Preferred Service Completion Times}

\author{Bahar \c{C}avdar \thanks{bcavdar@tamu.edu} \and Tu\u{g}\c{c}e I\c{s}{\i}k \thanks{tisik@clemson.edu}}
\date{%
    $^*$\small{Department of Engineering Technology and Industrial Distribution, Texas A\&M University, \\College Station, TX, USA }\\%
    $^\dagger$Department of Industrial Engineering, Clemson University, Clemson, SC, USA\\[2ex]%
    % \today
}

\maketitle

\begin{abstract}
{
Retailers use a variety of mechanisms to enable sales and delivery. A relatively new offering by companies is curbside pickup where customers purchase goods online, schedule a pickup time, and come to a pickup facility to receive their orders. To model this new service structure, we consider a queuing system where each arriving job has a preferred service completion time. Unlike most queuing systems, we make a strategic decision for when to serve each job based on their requested times and the associated costs. We assume that all jobs must be served before or on their requested time period, and the jobs are outsourced when the capacity is insufficient. Costs are incurred for jobs that are outsourced or served early. For small systems, we show that optimal capacity allocation policies are of threshold type. For general systems, we devise heuristic policies based on similar threshold structures. Our numerical study investigates the performance of the heuristics developed and shows the robustness of them with respect to several service parameters. Our results provide insights on how the optimal long-run average costs change based on the capacity of the system, the length of the planning horizon, cost parameters and the order pattern.

%Retailers use a variety of mechanisms to enable sales and delivery. A relatively new offering by companies is curbside pick up where customers purchase goods online, schedule a pick up time, and come to a pick up facility to receive their orders. To model this new service structure, we consider a queuing system where each arriving job has a preferred service completion time. Unlike most queuing systems where each job should be served as soon as possible, we make a strategic decision for when to serve each job based on their requested times and the associated costs. We assume that all jobs must be served before or on their requested time period, and the jobs are outsourced when the capacity is insufficient. Costs are incurred for both outsourced jobs and the jobs that are served early.  For small systems with short planning horizons and a single class of jobs, we show that optimal capacity allocation policies are of threshold type. For systems with arbitrary planning horizons, we devise heuristic policies based on similar threshold structures. Our numerical study investigates the performance of the heuristics developed and shows the robustness of them with respect to several service parameters. Our results also provide insights on how the optimal long-run average costs change based on the capacity of the system, length of the planning horizon, cost parameters and the order pattern of the customers. 
}
\end{abstract}
\section{Introduction}

Availability of online shopping options has drastically changed the retail of many consumer products from clothing to home appliances \citep{Chan2003}. Many traditionally brick and mortar retailers are now utilizing multiple channels to reach their customers \citep{Gallino2014}. Recent innovations in the retail industry has expanded the opportunities created by online stores to retail of perishable goods such as fresh produce and dairy products. Retailers like Walmart and Target now offer curbside pick up services. Walmart's online grocery pickup service is available in over 30 markets with more than 30,000 stock keeping units available for pickup.  Target recently launched its own pickup service application to be piloted in the twin cities area in Minnesota. In a similar vein, Amazon  has expanded its grocery retail services with two pilot pickup locations in the Seattle area. These innovations present a competitive alternative to another ongoing trend in the retail industry, namely same-day deliveries. On the other hand, companies are faced with new challenges due to the limitations in packaging and storage of highly perishable grocery items in the orders to be picked up. Walmart reports that 90\% of the online baskets contain fresh items like dairy, meat and produce \citep{Perez2016}. As a result, the planning and scheduling of order picking operations is critical for the success of these new services.

A typical grocery pickup service allows customers to go on a website, fill their virtual shopping carts with the items they would like to purchase and select a time slot for pick up before they check out. In all of the services mentioned above, the customer is presented with hourly time windows over the span of next few days. Depending on the capacity available at the location chosen for pickup, time slots that are too close to the time of purchase might become unavailable. At the selected pickup time, customers show up at the location they had selected and receive their packages. The service process described above has novel characteristics compared to traditional service schemes where customers are served on first-come-first serve basis or by appointment. First, customer requests need not be processed in the order they are received since each customer selects a preferred time slot for the completion of service (i.e., order pickup). Second, starting the service process as late as possible is advantageous to maintain the freshness of the perishable items purchased. In doing so, the company should keep any capacity restrictions in mind since the demand is unlikely to be distributed evenly over the time slots presented to the customers. Third, part of the service process can be completed in advance, since the company has more control over when to start the service compared to service operations where appointments are the norm.

Although the applications we consider here came to existence recently, a need for similar considerations may arise in other service systems where there is a gap between the time the service is requested and the time it is needed. Such service systems often operate through the scheduling of time windows to account for the preferences in the timing of service delivery. For instance, in delivery and maintenance services, companies may allow customers to select time windows rather than specific appointment times. These time constraints result in additional considerations of customer preferences and service quality for a given level of service capacity, leading to a higher level of complexity in operations planning. An overview of similar problems from the other domains is presented in the next section.% Even though same-day appointments for healthcare providers has been focus of recent research \cite{Green2008}, many patients decline the earliest available appointment times when offered\cite{Murray2000}.

This study considers a single queue of service requests in a system where customers have preferred service completion times. We assume that the service requests are received and served in discrete time periods where there is a rolling planning horizon of $K$ periods. In each time period $t$, a number of new service requests with preferred completion times, that are within the next $K$ periods, are observed according to a stochastic arrival process. The service capacity is constant over all time periods and at most $M$ requests can be served in any given period. We also assume that late fulfillment of requests are not allowed, hence there are two options as to when a job can be served: either start the service earlier than the requested time or serve on time. In the first case, we incur an early service cost $c_e$ per period. This cost  can be considered as a reflection of customer dissatisfaction, decay in the quality of the service, or holding cost of finished goods. In the second case, if the capacity is enough to serve all the jobs that are due in a period, they can be served with no additional cost. Otherwise, some jobs are fulfilled via overtime or outsourcing and a cost of $c_o$ per job is incurred. Once the jobs are served, the remaining requests are placed in a queue to be served later.

% we can note that it can be also overtime/outsourcing cost. we can also give examples of how early service is undesirable and so costly on our side. 

%In each time period, after serving all the jobs that are due the current period, a decision is to be made on which jobs (if any) from the queue should be served based on the remaining capacity. If the decision is to serve a job before its preferred start time, the early service delivery results in a cost of $c_{e}$ per each time period before the start time. This cost can be considered as a reflection of customer dissatisfaction, holding cost of finished goods, etc. For example, in a catering service, if the food is prepared before the requested time, it will not be as fresh at the time of serving, which reflects as decreased customer satisfaction implicitly measured by $c_{e}$. We assume all service can be completed within the time period it is initiated. The objective is to find a capacity allocation policy that results in the minimum long-run average cost.

The dimensionality of the problem described above rapidly increases with the number of time periods $K$ in the planning horizon and the maximum number of arrivals that can be observed in a period, which we denote by $A$. Since the problem size increases quickly, identifying the optimal policy is computationally challenging for instances of practical interest. Therefore, instead of searching for an optimal solution in the most general case, we propose a heuristic approach that finds the optimal solution for almost all instances in our extensive computational experiments. Our results provide insights on how the optimal policy changes under the following scenarios: (1) when the capacity of the system increases or decreases, (2) when the company offers different flexibilities to the customers in the form of different planning horizons, (3) when the early service and overtime costs change based on the nature of the problem. 

The structure of this paper is as follows. In Section 2, we present an overview of related literature. Then, we provide an MDP formulation for our capacity allocation problem in Section 3. In Section 4, we present our results on the structure of the optimal allocation policy. In Section 5, optimal capacity allocation policies are provided for small systems, and heuristic policies for larger systems are devised based on these results as well as the results provided in Section 4. We present our computational results evaluating the performance of the proposed heuristic policy in section 6. Finally in Section 7, a summary of our key results and insights are discussed.

\section{Literature Review}

Our work relates to several research problems related to dynamic capacity allocation, scheduling, and Markov decision processes. Dynamic resource allocation problems have been studied in a variety of contexts such as assembly lines, transportation, and service systems. There has been ample research in each of these areas where researchers investigated the effective use of resources in the face of demand uncertainty and associated costs. In a similar vein, our work focuses on the allocation of resources when the demand is random and capacity shortages are costly; however, it is unique in the sense that we assume jobs may arrive in the system earlier than their preferred time of service. 

In the literature, the idea that it might be preferable to postpone jobs until a certain time is often modeled using time windows, where a job can only be served after the beginning of a given time window and it needs to be completed before the end of the same time window. Consequently,  several papers considered scheduling of jobs with time windows on a single machine or multiple parallel machines. Hard time windows or earliness/tardiness constraints typically add to the complexity of scheduling problems, often making solutions intractable \citep{Gabrel1995}. As a result, many researchers chose to work with softer constraints where the undesirability of earliness or tardiness can be represented in terms of costs. \cite{Koulamas1996} considers the single machine scheduling problem with earliness and tardiness penalties and developed heuristic methods to sequence the jobs and optimally generate schedules given a predetermined sequence. \cite{Wan2002} focus on a similar single machine scheduling problem where jobs are to be served in distinct time windows. They propose a tabu search algorithm to minimize the total weighted earliness and tardiness. \cite{Monch2006} study the single machine scheduling problem minimizing earliness-tardiness with an upper limit on allowable total tardiness. Scheduling of a given set of jobs with distinct arrival times and distinct due dates on parallel machines in the presence of earliness and tardiness penalties is considered by \cite{Sivrikaya1999}. They develop and evaluate genetic algorithms for large problem instances. \cite{KedadSidhoum2008} study a similar problem and develop bounds on the optimal earliness-tardiness costs. In all of these works, it is assumed that the jobs that are to be scheduled are known in advance and a single, finite-horizon schedule sought. This is in contrast to dynamically allocating resources as jobs arrive while taking the desired timing of service into account as we aim to do in this paper.

%In addition to the studies that focus on minimizing earliness/tardiness penalties, there are also papers that consider scheduling of a set of order deliveries/jobs over a finite time horizon with the objective of maximizing the profits or revenue\cite{cc}. 

There has been little work on how to handle jobs with preferred service start time in queuing systems with random arrivals. \cite{Thiagarajan2005} study a setting that resembles the problem structure we consider the most. They consider a dynamic assembly job shop scheduling problem where jobs arrive following a Poisson process and the objective is to serve jobs as closely as possible to their due dates in order to avoid earliness costs that are due to accumulation of inventory and any penalties that are due to missing due dates. For this system, they evaluate the performance of a several dispatching rules via a simulation study. However, they do not provide any insights regarding the structure of the optimal scheduling policies, and their results do not include any comparisons with the optimal dispatching rules. Moreover, many papers that involve \emph{dynamic scheduling} focus on adjusting/repairing existing schedules as random events that relate to jobs or resources occur instead of scheduling jobs as they arrive or as time passes (for examples, see the review by \cite{Ouelhadj2009}). A few papers focused on \emph{completely reactive} scheduling mechanisms that use real-time dispatching rules to decide which job to process at an available machine. \cite{Holthaus2000} provide an overview of the existing dispatching rules. Based on these rules, they derive two new approaches to minimize maximum tardiness.  \cite{Aytug2005} provide a comprehensive summary of the existing research on reactive scheduling. Previous studies conclude that real-time dispatching rules are often myopic and fail to consider all the information related to the current state of the system such as arrival distribution and overtime probabilities. Thus, our work is distinct in that we are interested in allocating resources dynamically with the objective of optimizing long-run performance as the system state evolves.

In summary, the problem we investigate and the proposed solutions are different than the previous work in three major ways: (i) We provide a new, and to our knowledge, first MDP formulation for the capacity allocation problem in queues with earliness and tardiness costs. This formulation allows for extensions with respect to multiple arrival types and cost structures. (ii) Our model uses information available on the arrival process to inform capacity allocation decisions throughout the planning horizon. (iii) We provide structural results on the optimal allocation policies, propose a heuristic approach that finds near-optimal solutions, and provide a sensitivity analysis on both the resulting optimality gap and computational performance.

\section{Problem Formulation}

\label{Sec:ProbDef}

In our problem setting, customers arrive at each discrete time period with a service request to be fulfilled either in the current or a future period. That is, a customer arriving at time $t$ can request to be served at $t+j$, where $0\leq j < K$ where $K$ is the length of the planning horizon. Each service request can be completed within one period, thus all servers are available at the beginning of each period.  Once a request is made, there is no cancellation. % and jobs are non-preemptive. 

Our goal is to allocate the available capacity to the service requests waiting in the queue to maximize the long run average profit. Since we assume that customers do not leave the queue and are ultimately served, any revenue collected is independent of the allocation decisions and can be ignored. Therefore, we define our objective as the minimization of the long-run average total cost. 

The nature of this problem lends itself to be modeled as discrete-time infinite-horizon Markov decision process (MDP). In our MDP formulation, the decision epochs are immediately after customer requests arrive at each discrete time period. Since each service request can be completed in the same period it starts, the number of available servers at the beginning of a period is always $M$. Therefore, it is enough to track the customer requests in the queue at time $t$ to determine the state of the system. Let $X^{t}=x$ denote the state of the system at time $t$ and $x=\{x_{j}: 0\leq j < K\}$, where $x_{j}$ is the number of jobs in the queue that are due to be served $j$ periods later. We denote the set of all possible states by $\mathcal{X}$. 

We assume that, in each period $t$, number of customer arrivals with preferred service completions in period $t+j$ follows a known distribution with arrival rate $\lambda_j$ and these arrival processes are independent of each other for all $j$.  We further define the following quantities to specify the customer preferences for service delivery times over the planning horizon $K$:

\begin{itemize}
\item $\lambda$: overall arrival rate of customers.
\item $q_{j}$: probability that a customer arriving at $t$ requests a service starting at $t+j$, 
\begin{equation}
\lambda_{j}=q_{j}\lambda \quad \forall j \in\{0, \dots, K-1\}\label{Eq:lamb}
\end{equation}

\item $a_{j}^{t}$: number of customer arrivals at time $t$ requesting a service completion at time $t+j$.
\item $p_{j}(a)$: probability of observing $a$ customers who request a service completion time that is $j$ periods ahead.
\end{itemize}

Note that if $\lambda_{j}$'s are equal for all $j$'s, then customers preferences are evenly distributed over the periods $j \in\{0, \dots, K-1\}$ . However, in some cases customers may have a stronger preference for certain periods in the planning horizon. Such behavior can be captured by using different rate $\lambda_{j}$ for each $j$. 

In each period, the decision maker determines which jobs to serve from the queue. This decision can also include strategically idling the servers to avoid early service costs. Note that all jobs that are due the current period must be served even if the capacity is not enough. Early service is suboptimal unless there is remaining capacity after serving $x_{0}$ jobs, because we would be incurring both early service cost and overtime cost unnecessarily otherwise. We denote an action by $y=\{y_{j} :  0 \leq j < K \}$, where $y_{j}$ is the number of customers who requested to be served at $t+j$ but served at time $t$. Notice that we must have $y_{j}\leq x_{j}$ $\forall j$, hence the possible actions we can take are limited by the possible configurations of the customers in the queue. Therefore, the action space is identical to $\mathcal{X}$. The transition probabilities can be computed as follows:

\begin{equation}
\begin{array}{lr}
 P[X^{t+1}=x'| X^{t}=x, y] =P(\sigma_{K-1}^{t+1})\ \prod_{j=0}^{K-2} P(\omega_{j}^{t+1}) \label{Eq:form}
\end{array}
\end{equation}

where $\omega_{j}^{t+1}$ is the event $a_{j}^{t+1}=x_{j}'-(x_{ j+1}-y_{j+1})$ $\forall t$ and $0\leq j<K-1$, and $\sigma_{K-1}^{t+1}$ is the event  $a_{K-1}^{t+1}=x_{K-1}'$. The following example can provide a clearer understanding of the transition between the states.

\textbf{Example:} Consider a system with $K=3$ and suppose the system is at state $X^{t}=\{2, 0, 3\}$ at time $t$. There are two customers who want their service to start in the current period and three customers who want their service to start at $t+2$. Assume that we take action $y=\{2, 0, 1\}$. In the beginning of the next period, the queue has $\{0, 2, 0\}$ jobs before the arrivals are observed. So, the probability of reaching state $X^{t+1}=\{4, 2, 0\}$ is the probability of observing four arrivals requesting a service in period $t$, zero arrivals requesting a service in period $t+1$ and zero arrivals requesting a service in period $t+2$. Therefore we have 
$$P(X^{t+1}=\{4, 2, 0\}| X^{t}=\{2, 0, 3\}, y=\{2, 0, 1\})=P(a^{t+1}_{0}=4)P(a^{t+1}_{1}=0)P(a^{t+1}_{2}=0).$$ 

On the other hand, the probability of reaching the state $X^{t+1}=\{4, 0, 0\}$ is zero.

The total cost that is incurred in each period is the sum of the overtime and the early service costs. Note that the overtime and early service costs  cannot be incurred in the same period. We denote the cost incurred in a period $t$ by $u(x, y)$, where $x$ is the current state and $y=\{y_{j}: 0\leq j <K\}$ is the action taken. This cost is computed as 

\begin{equation}
u(x,y)=c_{o}(y_{0} - M)^{+}+c_{e}\sum_{j=1}^{K-1}jy_{j}.
\end{equation}

The probability of observing no customer arrivals is positive and non-diminishing for all periods in the time horizon. Thus, the state with $x_j=0$, $\forall j$ where there are no customers in the queue can be reached from any other state in finitely many steps with positive probability, and the MDP is unichain. Therefore, to find the policy that minimizes the long-run average cost we can solve the following optimality equation:

\begin{equation}
0=\min_{y\in\mathcal{X}}\bigg{\{}u(x,y)-g+\sum_ {x^{'}\in \mathcal{X}}P[x^{'}| x, y]h(x^{'})-h(x)\bigg{\}} \quad \forall x\in \mathcal{X}, \label{Eq:opt}
\end{equation}

where $g$ denotes the long-run average cost and $h(\cdot)$ is the bias \citep{puterman2014markov}. 

The curse of dimensionality due to large state and action spaces is an issue as in many other stochastic optimization problems with long planning horizons. To work with a finite state space, we truncate the arrival distributions so that in every period $t$, at most $A$ arrivals that are due in $t+j$ can be observed. Given these truncated distributions, the number of states and also the actions is $\Pi_{j=1}^{K}(jA+1)$, which still grows exponentially in the length of the planning horizon. This makes standard methods such as Policy Iteration computationally burdensome and impractical. Therefore, our efforts are in developing methods that can find near-optimal capacity allocation policies quickly.

\section{Monotonicity of the Optimal Policy}

Due to the multi-dimensionality of the state space and large number of allowable actions, the optimal server allocation policies for general systems is analytically intractable. In this section, we explore the properties of the optimal server allocation policies and show that the optimal policy is of threshold type for systems with short planning horizons. 
 
Consider a system with $K=2$ periods in the planning horizon, at most $A$ customer requests for $t+j$ observed in period $t$, and service capacity $M$. We define the dynamic programming value function as follows. Let $V_n(x)$ denote the minimal expected total cost if there are $n$ more periods left, where $x=\{x_0, x_1\}$ shows the number of customers in the system who want to be served in the current or the next period. Then the Bellman  equation associated with the value function is given as
 \begin{align}
 	V_n(x)=c_o(x_0-M)^++\min_{y_1\in Y_x}\{c_e y_1+\sum_{a_0\in\cal{A}}\sum_{a_1\in\cal{A}}p_0(a_0)p_1(a_1)V_{n-1}(a_0+x_1-y_1,a_1)\}
 	\label{Eq:value}
 \end{align}
 where $Y_x=\{0,1,\dots,\min\{x_1,(M-x_0)^+\}$ and $p_j(a_j)$'s $j\in\{0,\dots, K-1\}$ are the arrival probabilities for $a_j\in\cal{A}$=$\{0,1,\dots A\}$.
 
We would like to show that the optimal action $y_1^*$ is non-decreasing in the number of jobs that are due the next period. We define the following \emph{operators} for the value function $V_n(x)$:
 \begin{align*}
 T_o f(x)&=c_o(x_0-M)^++f(x)\\
 T_{es} f(x)&=\min_{y_1\in Y_x}\{c_e y_1+\sum_{a_0\in\cal{A}}\sum_{a_1\in\cal{A}}p_0(a_0)p_1(a_1)f(a_0+x_1-y_1,a_1)\}
 \end{align*}
where $T_o$ is defined to be the \emph{overtime operator} and $T_{es}$ is defined to be the \emph{early service operator}. Then, it suffices to show that these operators are increasing and componentwise convex to prove monotonicity of the optimal policy. The following two lemmas proofs of which are given in Appendix \ref{Sec:Proof} state these properties for $T_o$ and $T_{es}$.

\newtheorem{Lemma}{Lemma}
\begin{Lemma}
	\label{Lem:To}
For any function $f(x)$ that is increasing and  componentwise convex; the overtime operator $T_o$ preserves these properties. That is, $T_o f(x)$ is also increasing and componentwise convex.
\end{Lemma}

\begin{Lemma}
	\label{Lem:Tes}
For any function $f(x)$ that is increasing and  componentwise convex; the early service operator $T_{es}$ preserves these properties. That is, $T_{es} f(x)$ is also increasing and componentwise convex.
\end{Lemma}

Now we are ready to state our main result on the structure of the optimal capacity allocation policy in systems with short planning horizons. 

\newtheorem{Theorem}{Theorem}
\begin{Theorem}\label{main}
In a system with a planning horizon of $K=2$ periods and arbitrary service capacity, early service cost and overtime cost; the optimal number of jobs to be served early is monotone in the number of jobs that are due the next period, and thus the optimal capacity allocation policy is of threshold type. 
\end{Theorem}

\begin{proof}
 The dynamic programming value function $V_n(x)$ can be written as $T_o \circ T_{es}V_{n-1}(x)$. Thus, using the boundary condition $V_0(x)=0$ and induction, it follows that $V_n(x)$ is componentwise convex for every $n$ due to Lemmas 1 and 2. Then, the long-run average optimal value function is also componentwise convex \citep{Koole2007}.
 
 Recall that $V_n(x)=c_o(x_0-M)^{+} + \min_{y_1\in Y_x}\{c_{e}y_1+\sum_{a_0\in\cal{A}}\sum_{a_1\in\cal{A}} p_0(a_0)p_1(a_1) V_{n-1}(a_0+x_1-y_1,a_1)\}$. We define $W_{n-1}(x_1)=\sum_{a_0\in\cal{A}}\sum_{a_1\in\cal{A}} p_0(a_0)p_1(a_1) V_{n-1}(a_0+x_1-y_1,a_1)$. Then the value function can be rewritten as
\begin{equation}
\begin{split}
V_n(x)=c_o(x_0-M)^{+}+\min\bigg{\{}&W_{n-1}(x_1), c_{e}+W_{n-1}(x_1-1), 2c_{e}+W_{n-1}(x_1-2), \dots,\\ &\min\{x_1,(M-x_0)^+\}c_{e}+W_{n-1}(x_1-\min\{x_1,(M-x_0)^+\})\bigg{\}}.
\end{split} \nonumber
\end{equation}

Suppose $y_1^*$ is the smallest optimal action in some state $x=\{x_0,x_1\}$. Then, the following inequalities hold.
\begin{equation}
y_1^*c_e+W_{n-1}(x_1-y_1^*)< y_1c_e+W_{n-1}(x_1-y_1) \quad \forall  y_1 <y_1^*,
\end{equation}
\begin{equation}
y_1^*c_e+W_{n-1}(x_1-y_1^*) \leq y_1c_e+W_{n-1}(x_1-y_1) \quad \forall  y_1>y_1^*.
\end{equation}

Thus, we have
\begin{equation}
(y_1^*-y_1)c_e < W_{n-1}(x_1-y_1)-W_{n-1}(x_1-y_1^*) \quad \forall  y_1 <y_1^*, \label{ineq1}
\end{equation}
\begin{equation}
(y_1^*-y_1)c_e \leq W_{n-1}(x_1-y_1)-W_{n-1}(x_1-y_1^*) \quad \forall  y_1 >y_1^*.
\end{equation}

Now, also assume that there exists a state $x=\{x_0,x_1'\}$ with $x_1'>x_1$ where an action $y_1^{**}<y_1^*$ is optimal. This implies that the following inequalities also must hold:
\begin{equation}
(y_1^{**}-y_1)c_e \leq W_{n-1}(x_1'-y_1)-W_{n-1}(x_1'-y_1^{**}) \quad \forall  y_1 <y_1^{**},
\end{equation}
\begin{equation}
(y_1^{**}-y_1)c_e \leq W_{n-1}(x_1'-y_1)-W_{n-1}(x_1'-y_1^{**}) \quad \forall  y_1 >y_1^{**}.\label{ineq2}
\end{equation}

Since $y_1^{**}<y_1^*$, the following must be true:
\begin{equation}
(y_1^*-y_1^{**})c_e < W_{n-1}(x_1-y_1^{**})-W_{n-1}(x_1-y_1^{*}), \label{contr1}
\end{equation}
\begin{equation}
(y_1^{**}-y_1^*)c_e \leq W_{n-1}(x_1'-y_1^*)-W_{n-1}(x_1'-y_1^{**}) \label{contr2}
\end{equation}

due to \eqref{ineq1} and \eqref{ineq2} respectively. The inequalities \eqref{contr1} and \eqref{contr2} lead to a contradiction since we must have $W_{n-1}(x_1-y_1^{**})-W_{n-1}(x_1-y_1^*)\leq W_{n-1}(x_1'-y_1^{**})-W_{n-1}(x_1'-y_1^*) $ due to the componentwise convexity of $W_{n-1}(x_1)$. Thus the optimal action $y_1^*$ must be non-decreasing in $x_1$, and the optimal capacity allocation policy is of threshold type. 
\end{proof}

Even though the result given in Theorem \ref{main} is limited to systems with smaller planning horizons, monotonicity of the optimal policy provides valuable insights to tackle the high dimensionality of the problem. Next, we devise heuristic capacity allocation policies for systems with longer planning horizons using this structural result.

\section{Development of Heuristic Policies: Optimal Policy for Smaller Systems }

Due to the large size of the state and action spaces in our problem, algorithms such as Policy Iteration require large-scale computational resources and cannot be used efficiently to obtain optimal solutions. We propose a two-stage heuristic approach which is aimed to resolve the trade-off between solution quality and computation time. The first stage yields an initial policy and the second stage improves it. The policies that result from either stage can be used as heuristic based on the computation time requirements and the desired solution quality. In the second stage, we use \emph{one-step policy improvement,} a common technique to generate well-performing policies that can be computed quickly. One-step improvement heuristics have been widely used in many research domains such as appointment scheduling problems, routing in parallel queues, and organ allocation. The following are a few examples from the literature. \cite{Liu2010} use a one-step improvement heuristic to dynamically compute appointment schedules in the face of patient no-shows and cancellations. \cite{Feldman2014} use a similar method to develop appointment scheduling heuristics that take patient preferences into account. \cite{Opp2004} use dynamic programming policy improvement, among other methods, to develop index policies to minimize the cost of routing of warranty repairs to one of many available vendors. Heuristic methods for routing customers in a network of parallel service stations while minimizing waiting costs is developed using policy improvement by  \cite{TanikArgon2009}. Lastly, \cite{Zenios2000} use policy improvement to develop approximate solutions to the dynamic allocation problem for kidney transplants.

The decision rules used by our heuristics are derived using the threshold structure of the optimal policies. Understanding the relationships between the thresholds is easier when the planning horizon is short, especially when there are two periods. The decision for customers who want their service to start in the current period is obvious; they must be served immediately. If there is remaining capacity after these customers are served, one should decide whether additional customers should be served early or not. Based on our results in the previous section, we consider the following class of threshold policies to determine when to serve early. For a system with $K=2$, the threshold policy $S=(0, s_{1})$ serves all jobs that are due in the current period, and uses any remaining capacity to serve the jobs that are due in the next period as long as there are more than $s_{1}$ of them. For example, the policy $S= (0, 0)$ serves early whenever there is remaining capacity, while $S=(0, A)$ never serves early. The number of possible threshold policies is equal to $A+1$. The subsequent proposition provides the optimal policy of this type for a small but useful instance.

\newtheorem{Proposition}{Proposition}
\begin{Proposition}
The threshold policy $S^{*}$ given below is optimal for general cost parameters in a small system where $M=1$, $K=2$ and $A=2$.

\begin{equation}
\label{Eq:TH}
S^{*}=\begin{cases}
    (0, 0) \quad \textrm{if } c_{e}\frac{1+p_{0}(0)-p_{0}(0)p_{0}(1)-p_{0}^{2}(0)}{1-p_{0}^{2}(0)-p_{0}(0)p_{0}(1)} \leq c_{o},\\
    (0, 1) \quad \textrm{if }   c_{e} \leq c_{o} <   c_{e}\frac{1+p_{0}(0)-p_{0}(0)p_{0}(1)-p_{0}^{2}(0)}{1-p_{0}^{2}(0)-p_{0}(0)p_{0}(1)}  , \\
   (0, 2) \quad \textrm{if } c_{o} < c_{e}.\\
      \end{cases}
    \end{equation} 
    \label{Pro:opt}
\end{Proposition}
It is easy to show that the policy in Proposition \ref{Pro:opt} solves the optimality equation in (\ref{Eq:opt}). Therefore, we omit the proof for brevity.

The following proposition shows that if the early service cost is large enough, then the optimal policy never allows early service.

\begin{Proposition}
When $c_{o}\leq c_{e}$, regardless of the number of servers $M$ and the planning horizon $K$, the optimal threshold policy $S^{*}$ is always of the form $(0, A(K-1), A(K-2), \dots, A)$ (i.e., never serves early). \label{Pro:noserve}
\end{Proposition}

\begin{proof}
Consider a general system with $M$ servers and $K$ periods. If we serve a job $j$ periods ahead of the requested service time, the incurred early service cost is $jc_{e}$ for that job. However, if we keep the job in the queue until the requested service time, this job will be served in overtime if the capacity is exceeded, which will occur with some probability $p$. The expected overtime cost to serve this customer is $pc_{o}$, which is never greater than $jc_{e}$. Therefore, whenever $c_{o} \leq c_{e}$, the optimal action is never serving early. 
\end{proof}

Proposition \ref{Pro:noserve} determines the optimal threshold when $c_{o}\leq c_{e}$. However, we need to determine the optimal thresholds by solving the optimality equation (\ref{Eq:opt}) whenever this is not the case. Based on Propositions \ref{Pro:opt} and \ref{Pro:noserve}, the following Corollary directly follows.

\newtheorem{Corollary}{Corollary}
\begin{Corollary}
The optimal threshold policy in systems with $M=1$, $K=2$, and any maximum number of arrivals $A$ is as follows: 

\begin{equation}
S^{*}=\begin{cases}
    (0, 0) \quad \textrm{if } c_{e}\frac{1+p_{0}(0)-p_{0}(0)p_{0}(1)-p_{0}^{2}(0)}{1-p_{0}^{2}(0)-p_{0}(0)p_{0}(1)} \leq c_{o},\\
    (0, 1) \quad \textrm{if }   c_{e} \leq c_{o} <  c_{e}\frac{1+p_{0}(0)-p_{0}(0)p_{0}(1)-p_{0}^{2}(0)}{1-p_{0}^{2}(0)-p_{0}(0)p_{0}(1)}  , \\
   (0, A) \quad \textrm{if } c_{o} < c_{e}.\\
      \end{cases}
      \label{Eq:OptT}
    \end{equation} 
\end{Corollary}

The term $\phi=c_{o} \frac{1-p_{0}^{2}(0)-p_{0}(0)p_{0}(1)}{1+p_{0}(0)-p_{0}(0)p_{0}(1)-p_{0}^{2}(0)}$ can be interpreted as a measure of risk of incurring overtime cost for a job if we do not serve it early. Whenever $c_{o}\phi$ is higher than the for-sure cost in case of an early service, i.e., $c_{e}$, we serve the request early if the remaining capacity allows. 

As the number of servers and the length of the planning horizon increase, the structure of the optimal thresholds become too complex to be expressed in closed form. A simple heuristic for systems with longer planning horizons (i.e., $K>2$) can be devised by extending the threshold structure given in \eqref{Eq:OptT} in a rolling mechanism. Let $S=(s_{0}, s_{1}, \dots, s_{K-1})$ be a threshold policy for some $K$. For every $(j-1,j)$ pair where $1\leq j <K$, we compute the threshold $s_j$ sequentially as:

\begin{equation}
s_{j}=\begin{cases}
0 \quad \textrm{if } c_{e}\frac{1+p_{j-1}(0)-p_{j-1}(0)p_{j-1}(1)-p_{j-1}^{2}(0)}{1-p_{j-1}^{2}(0)-p_{j-1}(0)p_{j-1}(1)} \leq c_{o},\\
1 \quad \textrm{if }   c_{e} \leq c_{o} <   c_{e}\frac{1+p_{j-1}(0)-p_{j-1}(0)p_{j-1}(1)-p_{j-1}^{2}(0)}{1-p_{j-1}^{2}(0)-p_{j-1}(0)p_{j-1}(1)} ,\\
A \quad \textrm{if } c_{o} < c_{e}.\\
\end{cases}
\label{Eq:GenOptT}
\end{equation} 

Since we do not allow late service completions, we must serve all $x_0$ customers with requests due in the current period. Therefore, the threshold for $j=0$ is $s_0=0$, and $y_0=x_0$. Given a set of thresholds $s_j, j\in\{1,\dots,K-1\}$, the heuristic policy chooses actions  
\begin{equation*}
y_{j}=\min\{\max\{x_{j}-s_{j}^{*}, 0\}, M- \sum_{i<j}y_{j}\}
\end{equation*}

 for all $j$. In particular, after serving the excess customers who requested a service completion at $t+j$, the remaining capacity is decreased by $y_{j}$, and the process continues until we make a decision for all $j\in\{1,\dots,K-1\}$ or we run out of capacity.

The heuristic described above can be implemented for an arbitrary number of servers. Our preliminary computational results have shown that while the heuristic performs very well when $M$=1, it does not perform as well for larger number of servers. To address this, instead of the threshold policy in (\ref{Eq:OptT}), we use a local optimal threshold policy that takes service capacity into account when $M\geq 2$. We modify our heuristic as follows: given a problem instance, we consider an auxiliary problem with $K=2$ for each pair of time periods $t+j-1$ and $t+j$, $j\in\{1,\dots,K-1\}$, where all other parameters are identical to those of the original instance. For each new problem, we compute the average gain $g_{s_1}$ for decreasing thresholds $s_1$, where $(0,s_1)\in \{(0, A)$, $(0, A-1),\dots\}$ and we pick the largest threshold value $s_1^*$ that satisfies $g_{s_1^*}<g_{s_1^*-1}$. These thresholds are extended for the true planning horizon in the same way described earlier. We call the resulting policy the \emph{threshold heuristic} (TH). Considering the tradeoff between the solution quality and the computation time, we devise a second heuristic policy which we call \emph{threshold heuristic with 1-step improvement} (TH1S). TH1S determines an initial policy in the same way as TH, and performs a one-step policy improvement. The pseudo code of TH is given in Appendix \ref{App:heuristic}.

Since the heuristics are developed based on the optimality conditions of small systems, we know that they are optimal for the single-server system with two periods in the planning horizon. However, we do not have a theoretical guarantee for the solution quality in general systems. To observe the practical performance of our heuristics, computational experiments on many different scenarios were performed. In the next section, we present and discuss these experimental results.

\section{Computational Results}

In this section, we compare our heuristics with the optimal policy and the two other benchmark policies. These two policies, unlike our heuristics, do not make use of the problem structure. The first benchmark heuristic only serves the requests from the queue which are due in the current period, and never serves early. We call it \emph{do-nothing} (DN) policy. The second benchmark heuristic uses the do-nothing policy as the initial policy and performs one-step policy improvement. This heuristic is named \emph{do-nothing 1-step policy improvement} (DN1S). Do-nothing policy is also used as the initial policy to find the optimal solution using the policy iteration algorithm \citep{puterman2014markov}.

The scenarios we use to evaluate our heuristics are generated by varying the length of planning horizon $K$, number or servers $M$, arrival rate $\lambda$ and the maximum number of arrivals $A$. We also experiment on different customer order behavior. As briefly mentioned in Section \ref{Sec:ProbDef}, in some systems customers may be more likely to request service completion times that are early in the planning horizon. Such cases can occur if customers have shorter planning horizons than the service provider. In other cases, customers may be more likely to request a service completion that is late in the planning horizon, they may be equally likely to request any of the future time periods, or the timing of the requests might follow an arbitrary distribution. We describe the customer order behavior as the \emph{load} of the system. Different loads in the system can be captured by using a different arrival rates $\lambda_{j}$ for each period $t+j$. We formally define the four different load patterns we use, and compute the corresponding $q_{j}$ probabilities below.
\begin{enumerate}
\item \emph{Equal load systems (EL)}: Customers arriving at time $t$ are equally likely to request a service for any $t+j$, where $0\leq j<K$. 
\begin{equation}
q_{j}=\frac{1}{K},  \quad \forall j
\end{equation}

\item \emph{Front-loaded systems (FL)}: Customers arriving at time $t$ are more likely to request a service for $t+j_{1}$ than $t+j_{2}$ whenever $j_{1}< j_{2}$. 

\begin{equation}
q_{j}= \frac{(K-j)^{2}}{\sum_{i=1}^{K}i^{2}},  \quad \forall j
\end{equation}

\item \emph{Back-loaded systems (BL)}: Customers arriving at time $t$ are more likely to request a service for $t+j_{1}$ than $t+j_{2}$ whenever $j_{1}> j_{2}$. 

\begin{equation}
q_{j}=\frac{(j+1)^{2}}{\sum_{i=1}^{K}i^{2}},  \quad \forall j
\end{equation}

For the front-loaded and the back-loaded systems, we assume a quadratic change in $q_{j}$'s with respect to $j$. Corresponding $q_{j}$'s can be computed easily if the load over the planning horizon is changing at a different rate.

\item \emph{Arbitrarily-loaded systems (AL)}: Customer requests follow an arbitrary distribution over the planning horizon. The probabilities $q_{j}$ are randomly generated so that $\sum_{j=0}^{K-1}q_{j}=1$ and $q_{j}>0$ for all $j$.
\end{enumerate}

For the computational experiments, a plausible choice of the distribution for the number of arrivals observed is the truncated Poisson distribution. Once $q_{j}$'s are computed and the arrival rates $\lambda_j=\lambda q_j$ are determined for an overall arrival rate $\lambda$, truncated Poisson arrival probabilities can be computed according to the following equation:

\begin{equation}
p_{j}(a)=\frac{\frac{exp(\lambda_{j})\lambda_{j}^{a}}{a!}}{\sum_{i=0}^{A}\frac{exp(\lambda_{j})\lambda_{j}^{i}}{i!}},\quad \forall j, \forall a% \forall j \in\{0, \dots, K-1\}, \forall a \in\{0, \dots, A\} \label{Eq:p}
\end{equation}

We created our test scenarios by considering different values of the problem parameters: number of servers $M$, general arrival rate $\lambda$, maximum number of arrivals $A$, length of the planning horizon $K$, the cost ratio $c_{e}/c_{o}$ and the load of the system. For the sake of brevity, we only present the results that reveal important relations between the performance of the heuristics and the problem parameters. Parameter values in the experiments are given below.

\begin{table}[H]
\begin{center}
\begin{tabular}{r c l}
\hline
 $M$ &$\in$&   $\{1, 5\}$\\
 $A$ & $\in$ &      $\{1,	2,	3,	5,	10\}$\\
 $\lambda$ &=&$0.2A$\\
$K$ &$\in$   & $\{3, 4, 5\}$\\
$c_{o}$   &=&   $20$\\
$c_{e}$ &$\in$   & $\{5, 10\}$\\
\hline 
\end{tabular}
\end{center}
\end{table}

Some of these parameter combinations result in very large state spaces that exceed our computational power. Therefore, such combinations are omitted. All algorithms are coded in C programming language and computations are performed on Xeon E5520 machines. 

We present the optimal solution (Opt) and the heuristic solutions found by do-nothing policy (DN), do-nothing one-step policy improvement (DN1S), our threshold heuristic (TH) and our threshold heuristic with one-step policy improvement (TH1S). For each instance the optimal policy is computed using the policy iteration algorithm. In Tables \ref{Tab:load}--\ref{Tab:K5ce10}, we show the long-run average cost, running time (in seconds) and average percentage deviation from the optimal solution for each heuristic. When a heuristic finds the optimal solution, it is shown in bold in the tables. Notice that the run times include the computations of both the policy and its long-run average cost. The computation times for DN and TH are approximately how long it takes to evaluate the long-run average cost of a policy. 

The factors that affect the performance of these five solution methods are the load of the system, the length of the planning horizon, the number of servers, the ratio between the early service cost and the overtime cost and the maximum number of arrivals in a period. Therefore, we will organize our discussions based on these factors. 

To discuss how the performance of each algorithm changes based on the load of the system, we present the results corresponding to a system where $M$=1, $K$=4, $c_{e}$=5, $c_{o}$=20 with different load distributions in Table \ref{Tab:load}. For the same results,  we also present Figure \ref{Fig:load} to show the deviation of the heuristic solutions from the optimal in terms of the long-run average cost. Looking at Figure \ref{Fig:load}, the most striking observation is the change in the performance of the DN heuristic with respect to different loads of the system. In front-loaded systems, it is likely the that a large proportion of arriving requests will need to be served in the near future. Therefore, the DN heuristic that only serves the customer requests that are due the current period does not deviate from the optimal solution much especially when the number of servers is small. After serving the necessary jobs, which are due the current period, there is either no or very small service capacity left, which limits the suboptimal decisions that can be made by the DN heuristic. However, the situation is the opposite in the back-loaded systems. In a back-loaded system, a large number of requests are likely to arrive to be served later in the planning horizon. Therefore, after serving these jobs, most of the time there is still some capacity that may need to be allocated for early-service decisions. Since the DN heuristic is myopic, it cannot foresee the orders with approaching due dates which may result in overtime work later. When we have equal-load, the performance of DN is in between. These observations applies to other heuristics as well but at a smaller scale.

Regarding the computation time, TH1S and DN1S use approximately the same computation time since they both perform one-step policy improvement. TH and DN computation times are shorter and very close to each other since we only evaluate the initial policy. In this setting, we observe that our TH heuristic performs the same as DN1S heuristic in terms of the solution quality. Therefore, considering the tradeoff between the solution quality and the computation time, our TH and TH1S heuristics provide alternative solutions. More importantly our TH1S heuristic finds the optimum in all instances we considered in a significantly shorter time than the policy iteration algorithm.

\begin{table}[H]
\caption{Computational results on $M$=1, $K$=4, $c_{e}$=5, $c_{o}$=20.}
\label{Tab:load}
\begin{center}
\scalefont{0.8}
\begin{tabular}{l|l|rr|rr|rr|rr| rr} %l:left c:center r:right |:table lines
\cmidrule[1pt]{3-12}                  % 1pt is the thickness 3-10 is column number
\multicolumn{2}{c|}{}&\multicolumn{2}{c|}{\textbf{Opt}}&\multicolumn{2}{c|}{\textbf{DN}}&\multicolumn{2}{c|}{\textbf{DN1S}}&\multicolumn{2}{c|}{\textbf{TH}}&\multicolumn{2}{c}{\textbf{TH1S}}\\\midrule
&A & Cost & t(sec)& Cost & t(sec)& Cost & t(sec)& Cost & t(sec)& Cost & t(sec) \\\midrule\midrule
\multirow{3}{3mm}{\begin{sideways}{EL}\end{sideways}}
&1	&	0.18	&	0.02	&	0.26	&	0.00	&	0.19	&	0.01	&	0.19	&	0.00	&\textbf{0.18}	&	0.01	\\
&2	&	0.98	&	7.55	&	1.38	&	2.47	&	1.01	&	5.11	&	1.01	&	2.48	&\textbf{	0.98}	&	4.98	\\
&3	&	2.27	&	562.46	&	2.97	&	140.06	&	2.30	&	280.89	&	2.30	&	137.62	&	\textbf{2.27}	&	275.69	\\ \hline
\multirow{3}{3mm}{\begin{sideways}{FL}\end{sideways}}
&1	&	0.18	&	0.02	&	0.21	&	0.00	&\textbf{0.18}	&	0.01	&\textbf{0.18}	&	0.00	&\textbf{	0.18}	&	0.01	\\
&2	&	1.18	&	10.28	&	1.33	&	2.54	&\textbf{	1.18}	&	5.00	&\textbf{	1.18	}&	2.48	&\textbf{1.18} 	&	4.99	\\
&3	&	2.57	&	421.51	&	2.95	&	140.08	&	\textbf{2.57	}&	275.87	&\textbf{2.57	}&	140.53	&	\textbf{2.57}	&	281.01	\\\hline
\multirow{3}{3mm}{\begin{sideways}{BL}\end{sideways}}
&1	&	0.09	&	0.02	&	0.21	&	0.00	&	0.10	&	0.01	&	0.10	&	0.00	&	\textbf{0.09}	&	0.01	\\
&2	&	0.67	&	7.57	&	1.33	&	2.48	&	0.79	&	4.97	&	0.79	&	2.55	&\textbf{0.67}	&	5.11	\\
&3	&	1.78	&	421.21	&	2.95	&	139.99	&	1.92	&	274.71	&	1.92	&	137.27	&	\textbf{1.78	}&	274.63	\\\hline
\multirow{3}{3mm}{\begin{sideways}{AL}\end{sideways}}
&1	&	0.19	&	0.02	&	0.26	&	0.00	&\textbf{0.19}	&	0.01	&	\textbf{0.19	}&	0.00	&\textbf{	0.19}	&	0.01	\\
&2	&	1.01	&	7.58	&	1.37	&	2.48	&	1.05	&	5.11	&	1.05	&	2.48	&	\textbf{1.01	}&	4.99	\\
&3	&	2.31	&	553.96	&	2.97	&	138.81	&	2.37	&	275.44	&	2.37	&	137.09	&\textbf{2.31}	&	273.91	\\
 \hline\hline
\multicolumn{4}{c|}{\textbf{Average Deviation}} & 47\% & &5\% & &5\%& &0\% & \\
\hline\hline
\end{tabular}
\end{center}
\normalsize
\end{table}

%\multicolumn{4}{c|}{\textbf{Average Deviation from Optimality}} 

%% FIGURE SHOWING THE DEVIATIONS
\begin{figure}[H]
\centering
\includegraphics[width=0.85\textwidth]{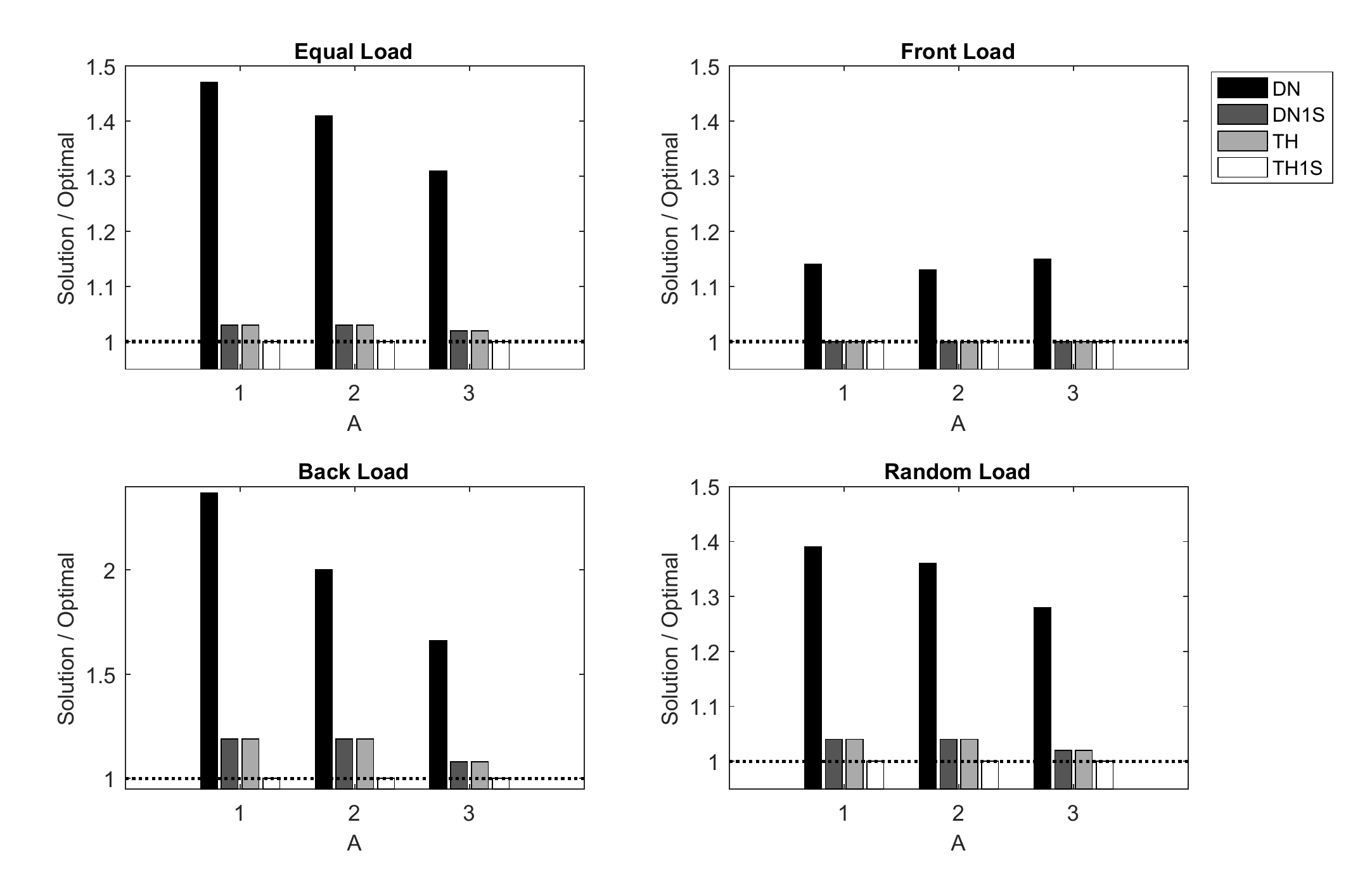}
\caption{Ratio between the heuristic cost and the optimal costs on $M$=1, $K$=4, $c_{e}$=5, $c_{o}$=20.}
\label{Fig:load}
\end{figure}

In Table \ref{Tab:ce10}, we present the results corresponding to a similar setting with only one difference; $c_{e}$ is increased from 5 to 10. Our TH1S heuristic again finds the optimal solution in all instances in much shorter time compared to the policy iteration algorithm. DN1S heuristic, which uses approximately the same computation time as TH1S, deviates by 5\% from the optimality on average. Comparing the results in Tables \ref{Tab:load} and \ref{Tab:ce10}, we see that DN performs better in Table \ref{Tab:ce10} on average. When $c_{e}$ increases, the optimal policy prescribes early service less frequently. Therefore, the solutions found by the DN heuristic, which never serves early, become less different than the optimal solutions. As a result, the average deviation from optimality for the DN policy decreases from 45\% to 22\% as $c_{e}$ increases. On the other hand, TH's performance becomes worse when the early service cost increases. Recall that TH is developed by extending the optimality conditions of a two-period system for longer planning horizons. 
Therefore, TH may consider some jobs that are due a future period to be more urgent than they actually are, and have a tendency to do more early service than the optimal solution. So, when $c_{e}$ increases, TH solutions deviate from the optimal solutions more. However, our TH1S heuristic still finds optimal solutions in reasonable computation times.

%%all same, ce is larger:
%%%
\begin{table}[H]
\caption{Computational results on $M$=1, $K$=4, $c_{e}$=10, $c_{o}$=20.}
\label{Tab:ce10}
\begin{center}
\scalefont{0.8}
\begin{tabular}{l|l|rr|rr|rr|rr| rr} %l:left c:center r:right |:table lines
\cmidrule[1pt]{3-12}                  % 1pt is the thickness 3-10 is column number
\multicolumn{2}{c|}{}&\multicolumn{2}{c|}{\textbf{Opt}}&\multicolumn{2}{c|}{\textbf{DN}}&\multicolumn{2}{c|}{\textbf{DN1S}}&\multicolumn{2}{c|}{\textbf{TH}}&\multicolumn{2}{c}{\textbf{TH1S}}\\\midrule
&A & Cost & t(sec)& Cost & t(sec)& Cost & t(sec)& Cost & t(sec)& Cost & t(sec) \\\midrule\midrule
\multirow{3}{3mm}{\begin{sideways}{EL}\end{sideways}}
&1 & 0.21 & 0.02   & 0.26 & 0.00 & \textbf{0.21} & 0.01      & 0.22          & 0.00      & \textbf{0.21} & 0.01   \\
&2 & 1.13 & 9.96   & 1.38 & 2.46          & 1.20          & 5.05      & 1.24          & 2.50      & \textbf{1.13} & 5.05   \\
&3 & 2.55 & 549.31 & 2.97 & 137.15        & 2.63          & 280.81    & 2.71          & 140.09    & \textbf{2.55}        & 280.42 \\
\hline
\multirow{3}{3mm}{\begin{sideways}{FL}\end{sideways}}
&1 & 0.19 & 0.02   & 0.21 & 0.00          & \textbf{0.19} & 0.01      & \textbf{0.19} & 0.00      & \textbf{0.19}         & 0.01   \\
&2 & 1.23 & 10.02  & 1.33 & 2.49          & 1.24          & 4.97      & 1.24          & 2.47      & \textbf{1.23}         & 4.98   \\
&3 & 2.77 & 549.55 & 2.95 & 137.22        & 2.78          & 274.73    & 2.79          & 139.28    & \textbf{2.77}          & 279.06 \\
\hline
\multirow{3}{3mm}{\begin{sideways}{BL}\end{sideways}}
&1 & 0.13 & 0.02   & 0.21 & 0.00          & 0.14          & 0.01      & 0.17          & 0.00      &\textbf{0.13}        & 0.01   \\
&2 & 0.95 & 9.97   & 1.33 & 2.47          & 1.12          & 5.06      & 1.27          & 2.47      &\textbf{0.95}        & 4.98   \\
&3 & 2.32 & 547.71 & 2.95 & 136.76        & 2.53          & 279.98    & 2.77          & 140.35    & \textbf{2.32}          & 280.65 \\
\hline
\multirow{3}{3mm}{\begin{sideways}{AL}\end{sideways}}
&1 & 0.21 & 0.02   & 0.26 & 0.00          & 0.23          & 0.01      & 0.23          & 0.00      &\textbf{ 0.21}         & 0.01   \\
&2 & 1.15 & 10.01  & 1.37 & 2.48          & 1.24          & 5.10      & 1.29          & 2.49      &\textbf{ 1.15}          & 5.01   \\
&3 & 2.59 & 552.62 & 2.97 & 137.56        & 2.71          & 281.03    & 2.79          & 136.61    & \textbf{2.59} & 268.79 \\ \hline\hline
\multicolumn{4}{c|}{\textbf{Average Deviation}} & 22\% & &5\% & &11\%& &0\% & \\
\hline\hline
\end{tabular}
\end{center}
\normalsize
\end{table}

The number of servers in the system also directly affects how the heuristics perform. As the capacity of the system increases, the number of jobs served in overtime decreases. If the capacity is indeed very large relative to the number of arrivals, then each service request can be served on its requested time. In such systems, DN and DN1S heuristics are expected to perform well since they do not consider early service. The computational results in Table \ref{Tab:Mlarge} present a case where the server capacity is abundant. The only parameter change from the previous scenario is increasing the number of servers to five. As it is seen in the table, DN, DN1S and TH1S find the optimal solution in all instances, while TH has a small deviation. These results support our expectations that the quality of the solutions found by DN and DN1S increase as the service capacity increases.
%Based on these results, we expect an increase in the quality of the solutions found by DN and DN1S as the service capacity increases.

\newcommand{\NA}{---}

%%LARGE M
\begin{table}[H]
\caption{Computational results on $M$=5, $K$=4, $c_{e}$=10, $c_{o}$=20.}
\label{Tab:Mlarge}
\begin{center}
\scalefont{0.8}
\begin{tabular}{l|l|rr|rr|rr|rr| rr} %l:left c:center r:right |:table lines
\cmidrule[1pt]{3-12}                  % 1pt is the thickness 3-10 is column number
\multicolumn{2}{c|}{}&\multicolumn{2}{c|}{\textbf{Opt}}&\multicolumn{2}{c|}{\textbf{DN}}&\multicolumn{2}{c|}{\textbf{DN1S}}&\multicolumn{2}{c|}{\textbf{TH}}&\multicolumn{2}{c}{\textbf{TH1S}}\\\midrule
&A & Cost & t(sec)& Cost & t(sec)& Cost & t(sec)& Cost & t(sec)& Cost & t(sec) \\\midrule\midrule
\multirow{3}{3mm}{\begin{sideways}{EL}\end{sideways}}
& 1 & 0.00 & 0.00   & \textbf{0.00} & 0.00   & \textbf{0.00} & 0.00   & 0.09          & 0.00   & \textbf{0.00} & 0.01   \\
 & 2 & 0.00 & 5.38   & \textbf{0.00} & 2.39   & \textbf{0.00} & 5.29   & 0.04          & 2.48   & \textbf{0.00} & 5.27   \\
 & 3 & 0.00 & 428.78 & \textbf{0.00} & 138.01 & \textbf{0.00} & 277.94 & 0.01          & 137.52 & \textbf{0.00} & 279.05 \\ \hline
\multirow{3}{3mm}{\begin{sideways}{FL}\end{sideways}}
 & 1 & 0.00 & 0.00   & \textbf{0.00} & 0.00   & \textbf{0.00} & 0.00   & 0.02          & 0.00   & \textbf{0.00} & 0.01   \\
 & 2 & 0.00 & 5.60   & \textbf{0.00} & 2.47   & \textbf{0.00} & 5.23   & 0.01          & 2.48   & \textbf{0.00} & 5.27   \\
 & 3 & 0.00 & 422.82 & \textbf{0.00} & 136.86 & \textbf{0.00} & 279.99 & \textbf{0.00} & 136.85 & \textbf{0.00} & 277.96 \\\hline
\multirow{3}{3mm}{\begin{sideways}{BL}\end{sideways}}
 & 1 & 0.00 & 0.00   & \textbf{0.00} & 0.00   & \textbf{0.00} & 0.00   & 0.15          & 0.00   & \textbf{0.00} & 0.01   \\
 & 2 & 0.00 & 5.61   & \textbf{0.00} & 2.48   & \textbf{0.00} & 5.33   & 0.09          & 2.50   & \textbf{0.00} & 5.33   \\
 & 3 & 0.00 & 422.65 & \textbf{0.00} & 136.97 & \textbf{0.00} & 279.61 & 0.04          & 139.95 & \textbf{0.00} & 283.31 \\ \hline
\multirow{3}{3mm}{\begin{sideways}{AL}\end{sideways}}
& 1 & 0.00 & 0.00   & \textbf{0.00} & 0.00   & \textbf{0.00} & 0.00   & 0.08          & 0.00   & \textbf{0.00} & 0.01   \\
 & 2 & 0.00 & 5.73   & \textbf{0.00} & 2.54   & \textbf{0.00} & 5.34   & 0.04          & 2.47   & \textbf{0.00} & 5.25   \\
 & 3 & 0.00 & 433.93 & \textbf{0.00} & 140.55 & \textbf{0.00} & 278.91 & 0.01          & 134.49 & \textbf{0.00} & 269.88 \\
\hline\hline
\multicolumn{4}{c|}{\textbf{Average Deviation}} & 0\% & &0\% & &\NA& &0\% & \\
\hline\hline
\end{tabular}
\end{center}
\normalsize
\end{table}

In the computational experiments, the length of the planning horizon $K$ emerged as the most influential problem parameter on how the heuristics perform. The computational challenge imposed by longer planning horizon results in longer computation times for all solution methods and a larger deviation from the optimal solution for DN, DN1S and TH, while our TH1S heuristic still finds the optimal solution in all cases. To analyze the effect of the length of the planning horizon, we present the results in Tables \ref{Tab:K3ce10}, \ref{Tab:K5ce5} and \ref{Tab:K5ce10}. %The parameter settings in these tables are similar to before with only different $K$ values so that we can see its effect on different cases. 

\begin{table}[H]
\caption{Computational results on $M$=1, $K$=3, $c_{e}$=10, $c_{o}$=20.}
\label{Tab:K3ce10}
\begin{center}
\scalefont{0.8}
\begin{tabular}{l|l|rr|rr|rr|rr| rr} %l:left c:center r:right |:table lines
\cmidrule[1pt]{3-12}                  % 1pt is the thickness 3-10 is column number
\multicolumn{2}{c|}{}&\multicolumn{2}{c|}{\textbf{Opt}}&\multicolumn{2}{c|}{\textbf{DN}}&\multicolumn{2}{c|}{\textbf{DN1S}}&\multicolumn{2}{c|}{\textbf{TH}}&\multicolumn{2}{c}{\textbf{TH1S}}\\\midrule
&A & Cost & t(sec)& Cost & t(sec)& Cost & t(sec)& Cost & t(sec)& Cost & t(sec) \\\midrule\midrule
\multirow{4}{3mm}{\begin{sideways}{EL}\end{sideways}}
&1&0.20  & 0.00    & 0.23  & 0.00    & \textbf{0.20}  & 0.00    & \textbf{0.20}  & 0.00    & \textbf{0.20}  & 0.00    \\
&2&1.16  & 0.01    & 1.36  & 0.00    & 1.19           & 0.00    & 1.19           & 0.00    & \textbf{1.16}  & 0.00    \\
&5&6.81  & 13.81   & 7.36  & 3.42    & 6.86           & 6.88    & 6.86           & 3.43    & \textbf{6.81}  & 6.87    \\
&10&22.09 & 4137.11 & 22.71 & 1035.06 & 22.20          & 2098.41 & 22.20          & 1039.98 & \textbf{22.09} & 2078.11 \\
\hline
\multirow{4}{3mm}{\begin{sideways}{FL}\end{sideways}}
&1&0.16  & 0.00    & 0.17  & 0.00    & \textbf{0.16}  & 0.00    & \textbf{0.16}  & 0.00    & \textbf{0.16}  & 0.00    \\
&2&1.23  & 0.01    & 1.30  & 0.00    & \textbf{1.23}  & 0.00    & 1.24           & 0.00    & \textbf{1.23}  & 0.00    \\
&5&7.16  & 13.93   & 7.35  & 3.45    & 7.17           & 6.83    & 7.20           & 3.43    & \textbf{7.16}  & 6.88    \\
&10&22.23 & 3109.00 & 22.71 & 1037.87 & \textbf{22.23} & 2073.43 & \textbf{22.23} & 1038.35 & \textbf{22.23} & 2073.04 \\
\hline
\multirow{4}{3mm}{\begin{sideways}{BL}\end{sideways}}
&1&0.12  & 0.00    & 0.17  & 0.00    & \textbf{0.12}  & 0.00    & \textbf{0.12}  & 0.00    & \textbf{0.12}  & 0.00    \\
&2&0.95  & 0.01    & 1.30  & 0.00    & 1.05           & 0.00    & 1.05           & 0.00    & \textbf{0.95}  & 0.00    \\
&5&6.46  & 13.60   & 7.35  & 3.38    & 6.59           & 6.78    & 6.59           & 3.51    & \textbf{6.46}  & 7.03    \\
&10&21.94 & 4235.91 & 22.71 & 1059.57 & 21.96          & 2072.61 & 21.98          & 1061.82 & \textbf{21.94} & 2123.15 \\
\hline
\multirow{4}{3mm}{\begin{sideways}{AL}\end{sideways}}
&1&0.20  & 0.00    & 0.22  & 0.00    & \textbf{0.20}  & 0.00    & \textbf{0.20}  & 0.00    & \textbf{0.20}  & 0.00    \\
&2&1.19  & 0.01    & 1.35  & 0.00    & 1.24           & 0.00    &1.24  & 0.00    & \textbf{1.19}  & 0.00    \\
&5&6.90  & 13.74   & 7.36  & 3.42    & 6.98           & 6.99    & 6.98           & 3.43    & \textbf{6.90}  & 6.87    \\
&10&22.11 & 3107.54 & 22.71 & 1037.73 & 22.12          & 2124.55 & 22.12          & 1038.37 & \textbf{22.11} & 2077.91 \\ \hline\hline
\multicolumn{4}{c|}{\textbf{Average Deviation}} & 12\% & &1\% & &1\%& &0\% & \\
\hline\hline
\end{tabular}
\end{center}
\normalsize
\end{table}

%%%

\begin{table}[H]
\caption{Computational results on $M$=1, $K$=5, $c_{e}$=5, $c_{o}$=20.}
\label{Tab:K5ce5}
\begin{center}
\scalefont{0.8}
\begin{tabular}{l|l|rr|rr|rr|rr| rr} %l:left c:center r:right |:table lines
\cmidrule[1pt]{3-12}                  % 1pt is the thickness 3-10 is column number
\multicolumn{2}{c|}{}&\multicolumn{2}{c|}{\textbf{Opt}}&\multicolumn{2}{c|}{\textbf{DN}}&\multicolumn{2}{c|}{\textbf{DN1S}}&\multicolumn{2}{c|}{\textbf{TH}}&\multicolumn{2}{c}{\textbf{TH1S}}\\\midrule
&A & Cost & t(sec)& Cost & t(sec)& Cost & t(sec)& Cost & t(sec)& Cost & t(sec) \\\midrule\midrule
\multirow{2}{3mm}{\begin{sideways}{EL}\end{sideways}}
& 1 & 0.18 & 3.49     & 0.28      & 1.12    & 0.19      & 2.27    & 0.19      & 1.12    & \textbf{0.18} & 2.28    \\
 & 2 & 0.92 & 9620.01  & 1.39      & 3203.87 & 1.00      & 6316.34 & 1.00      & 3219.64 & \textbf{0.92} & 6428.59 \\
\hline
\multirow{2}{3mm}{\begin{sideways}{FL}\end{sideways}}& 1 & 0.19 & 3.46     & 0.24      & 1.11    & 0.20      & 2.31    & 0.20      & 1.12    & \textbf{0.19} & 2.26    \\
 & 2 & 1.14 & 12842.10 & 1.36      & 3216.93 & 1.15      & 6443.10 & 1.15      & 3163.39 & \textbf{1.14} & 6338.47 \\\hline
\multirow{2}{3mm}{\begin{sideways}{BL}\end{sideways}} & 1 & 0.09 & 3.48     & 0.24      & 1.12    & 0.14      & 2.26    & 0.14      & 1.14    & \textbf{0.09} & 2.30    \\
 & 2 & 0.64 & 12739.88 & 1.36      & 3165.59 & 0.89      & 6323.69 & 0.89      & 3223.14 & \textbf{0.64} & 6327.60 \\\hline
\multirow{2}{3mm}{\begin{sideways}{AL}\end{sideways}}& 1 & 0.18 & 3.42     & 0.28      & 1.10    & 0.20      & 2.43    & 0.20      & 1.12    & \textbf{0.18} & 2.28    \\
 & 2 & 0.93 & 9673.41  & 1.39      & 3216.73 & 1.04      & 6998.82 & 1.04      & 3164.44 & \textbf{0.93} & 6296.38 \\
\hline\hline
\multicolumn{4}{c|}{\textbf{Average Deviation}} & 67\% & &17\% & &17\%& &0\% & \\
\hline\hline
\end{tabular}
\end{center}
\normalsize
\end{table}

%%%

\begin{table}[H]
\caption{Computational results on $M$=1, $K$=5, $c_{e}$=10, $c_{o}$=20.}
\label{Tab:K5ce10}
\begin{center}
\scalefont{0.8}
\begin{tabular}{l|l|rr|rr|rr|rr| rr} %l:left c:center r:right |:table lines
\cmidrule[1pt]{3-12}                  % 1pt is the thickness 3-10 is column number
\multicolumn{2}{c|}{}&\multicolumn{2}{c|}{\textbf{Opt}}&\multicolumn{2}{c|}{\textbf{DN}}&\multicolumn{2}{c|}{\textbf{DN1S}}&\multicolumn{2}{c|}{\textbf{TH}}&\multicolumn{2}{c}{\textbf{TH1S}}\\\midrule
&A & Cost & t(sec)& Cost & t(sec)& Cost & t(sec)& Cost & t(sec)& Cost & t(sec) \\\midrule\midrule
\multirow{2}{3mm}{\begin{sideways}{EL}\end{sideways}}
& 1 & 0.22 & 4.53     & 0.28      & 1.12    & \textbf{0.22}       & 2.27    & 0.26      & 1.13    & \textbf{0.22} & 2.28    \\
 & 2 & 1.11 & 12779.62 & 1.39      & 3204.71 & 1.21      & 6315.96 & 1.32      & 3164.02 & \textbf{1.11} & 6331.42 \\
\hline
\multirow{2}{3mm}{\begin{sideways}{FL}\end{sideways}}
 & 1 & 0.21 & 4.54     & 0.24      & 1.11    &\textbf{0.21}      & 2.27    & 0.21      & 1.13    & \textbf{0.21} & 2.31    \\
 & 2 & 1.22 & 12727.29 & 1.36      & 3157.63 & 1.24      & 6300.64 & 1.24      & 3188.95 & \textbf{1.22} & 6369.52 \\\hline
\multirow{2}{3mm}{\begin{sideways}{BL}\end{sideways}} 
 & 1 & 0.15 & 4.56     & 0.24      & 1.12    & 0.20      & 2.30    & 0.25      & 1.12    & \textbf{0.15} & 2.27    \\
 & 2 & 0.96 & 12622.16 & 1.36      & 3160.62 & 1.15      & 6454.27 & 1.57      & 3153.15 & \textbf{0.96} & 6312.36 \\\hline
\multirow{2}{3mm}{\begin{sideways}{AL}\end{sideways}}
& 1 & 0.22 & 4.56     & 0.28      & 1.12    & 0.25      & 2.30    & 0.27      & 1.12    & \textbf{0.22} & 2.27    \\
 & 2 & 1.12 & 12704.22 & 1.39      & 3157.67 & 1.24      & 6436.90 & 1.40      & 3155.02 & \textbf{1.12} & 6294.66 \\
\hline\hline
\multicolumn{4}{c|}{\textbf{Average Deviation}} & 30\% & &12\% & &28\%& &0\% & \\
\hline\hline
\end{tabular}
\end{center}
\normalsize
\end{table}

In Table \ref{Tab:K3ce10}, the results are presented for when there are only three periods in the planning horizon. Compared to the scenario when $K$=4 in Table \ref{Tab:ce10}, we see that all heuristics perform better when $K$=3. On the other hand, when there are five periods, the average quality of the solutions found by DN, DN1S and TH decreases significantly. This can be seen by comparing the corresponding results in Tables \ref{Tab:load} and \ref{Tab:K5ce5} and Tables \ref{Tab:K3ce10} and \ref{Tab:K5ce10}. These results are not very surprising for DN and TH policies since they do not make use of a policy improvement step. As $K$ increases, DN's myopic behavior that avoids early service and TH's greedy behavior to avoid overtime become more dominant. We also observe that even if DN1S performs a policy improvement, it cannot recover from the initial myopic policy much, yet our TH1S heuristic finds the optimal solution in all cases. 

In systems with shorter planning horizons as in Table \ref{Tab:K3ce10}, if the solution quality is of utmost importance, TH1S becomes the best heuristic to use. If small deviations from the optimality can be allowed for the sake of shorter computation time, then TH algorithm can be used. On the other hand, for larger planning horizons, DN, DN1S and TH heuristics are less competitive and TH1S remains as the only option which can find the optimal solution in all cases in a much shorter amount of time than the policy iteration.

\section{Conclusion}

In this paper, we study a capacity allocation problem in a queuing system where customers specify a desired service completion time. To our knowledge, this is the first study to analyze the structure of the optimal policy in such queuing systems and develop heuristic solutions that are based on the characterization of the optimal solution.

When customers have preferred service times, the customer satisfaction is measured by the service provider's ability to serve the jobs as close as possible to their requested times. Service can be completed either (1) early, (2) on time using the regular capacity or (3) on time by outsourcing or increasing the capacity temporarily. In the cases of (1) and (3), the service provider incurs early service cost and overtime cost respectively. The problem is to determine how to allocate the capacity, which may include idling servers strategically, to serve the jobs waiting in the queue at the minimum cost.

We formulate this problem as a Markov Decision Process (MDP). We show that the optimal policy is of threshold type for short planning horizons and characterize the optimal policy for small systems. Since the problem size increases exponentially in the length of the planning horizon, we develop heuristics that can find near-optimal solutions in reasonable computation times. Indeed, one of our proposed heuristics finds the optimal solution in almost all test instances.

Our results provide insights about how the optimal solution behaves in different scenarios. Certain combinations of problem parameters impose relatively easier scenarios to find good solutions. Such cases occur at the extreme realizations of the parameters that are directly affecting the cost of an action, i.e., early service cost, overtime cost and the server capacity. For example, when the ratio between the early service cost and the overtime cost is very large or small, optimal policy is similar to either \emph{never serve early} policy or \emph{serve early whenever possible} policy. However, when the ratio is moderate, a well-performing policy is harder to identify. A similar case is when the server capacity is either very scarce or abundant. In the first case, we rarely have remaining capacity after serving the requests that are due the current period, and we cannot consider the early service option. Therefore, the allocation decisions become very limited, which bounds the deviation from the optimal solution. When the capacity is large, almost all jobs can be served at their requested times without needing any overtime work. Moderate service capacity, however, is more difficult to deal with. On the other hand, the parameters that determine the size of the state space directly do not have such an effect. As the maximum number of arrivals that can be observed in a period and the length of the planning horizon increase, solving the problem becomes substantially more difficult and the benchmark heuristics are left with poor performance. Yet our proposed heuristic finds the optimal solution in almost all test instances.

The presented problem opens discussion to other interesting research questions. In this study, we assume that all customers are of the same type. Companies may assign a higher priority to their more loyal customers or to the ones with higher demand rates. Future work could focus on a variant of this problem for different customer types. Relaxing the assumption that each service can be completed within one period would be another extension. In addition, a similar problem can also be studied from the customer's perspective. For instance, when late service completions are allowed, if waiting in the queue is costly but increases the priority of a customer, the game theoretical problem of when to join the queue can be of interest.

\FloatBarrier
\bibliographystyle{ormsv080}
\bibliography{RevisedReferences}
\newpage 

\appendix

\section{Proofs of the Lemmas}
\label{Sec:Proof}

Proof of Lemma 1 is as follows.
\begin{proof}
	Suppose we have a function $f(x)$ that is increasing and componentwise convex. We first show that $T_o f(x)$ is also increasing. 
	
	Since the $f(x)$ is increasing and componentwise convex, we have
	\begin{align*}
	f(x)&\leq f(x+e_1+e_2)\\
	2f(x+e_i)&\leq f(x)+f(x+2e_i) \ \forall i\in\{1,2\}
	\end{align*}
	where $e_i$ is the unit vector where the  $i^{th}$ entry is one and all other entries are zero. Then, it follows that $T_o f(x)$ is increasing, i.e.,
	\begin{align}
	c_o(x_0-M)^++f(x)\leq c_o(x_0+1-M)^++f(x+e_1+e_2),
	\end{align}
	since $c_o(x_0-M)^+\leq c_o(x_0+1-M)^+$. 
	
	Similarly, $T_o f(x)$ is componentwise convex, i.e.,
	\begin{align}
	2f(x+e_1)+2c_o(x_0+1-M)^+\leq f(x)+c_o(x_0-M)^{+}+f(x+2e_1)+c_o(x_o+2-M)^+
	\end{align}
	and
	\begin{align}\label{trivial1}
	2f(x+e_2)+2c_o(x_0-M)^+\leq f(x)+c_o(x_0-M)^{+}+f(x+2e_2)+c_o(x_0-M)^+
	\end{align}
	since $2(x_0+1-M)^+\leq (x_0-M)^++(x_0+2-M)^+$ holds for all values of $x_0$. (The inequality \eqref{trivial1} follows trivially.)
	
	\end{proof}

Proof of Lemma 2 is given below.
\begin{proof}
We assume that $f(x)$ preserves the same properties as in Lemma 1. First, we will show that $T_{es} f(x)$ is increasing, i.e., $T_{es}f(x)\leq T_{es}f(x+e_1+e_2)$ where 

\begin{equation}
 T_{es} f(x)=\min_{y_1\in Y_{x^{1}}}\{c_e y_1+\sum_{a_0\in\cal{A}}\sum_{a_1\in\cal{A}}p_0(a_0)p_1(a_1)f(a_0+x_1-y_1,a_1)\},
 \label{Eq:first}
\end{equation}
\begin{equation}
\label{Eq:last}
 T_{es} f(x+e_1+e_2)=\min_{y_1\in Y_{x^{2}}}\{c_e y_1+\sum_{a_0\in\cal{A}}\sum_{a_1\in\cal{A}}p_0(a_0)p_1(a_1)f(a_0+x_1+1-y_1,a_1)\}.
\end{equation}

$Y_{x^1}=\{0,1,\dots, \min\{x_1, (M-x_0)^+\}\}$ and $Y_{x^2}=\{0,1,\dots, \min\{x_1+1, (M-x_0-1)^+\}\}$ are the decision sets for $y_1$. Note that for any given value $y_1$, the following holds since the function $f(x)$  is increasing:
\begin{align*}
c_ey_1+\sum_{a_0\in\cal{A}}\sum_{a_1\in\cal{A}}p_0(a_0)p_1(a_1)f(a_0+x_1-y_1, a_1)\leq c_ey_1+\sum_{a_0\in\cal{A}}\sum_{a_1\in\cal{A}}p_0(a_0)p_1(a_1)f(a_0+x_1+1-y_1, a_1),
\end{align*}

and thus it is easy to see that $T_{es}f(x)\leq T_{es}f(x+e_1+e_2)$ holds when $Y_{x^2}\subseteq Y_{x^1}$. Suppose that we have $Y_{x^1}\subset Y_{x^2}$. Then we must also have $\min\{x_1,(M-x_0)^+\}=x_1$ and $\min\{x_1+1, (M-x_0+1)^+\}=x_1+1$. Therefore, the set $Y_{x^2}$ differs from the set $Y_{x^1}$ by one element, and $Y_{x^2}$=$Y_{x^1}\cup \{x_1+1\}$. Let the minimizers for $T_{es}f(x)$ and $T_{es}f(x+e_1+e_2)$ be denoted by $y_1^{1*}$ and $y_1^{2*}$, respectively, and assume that $y_1^{2*}=x_1+1$. Based on the optimality of $y_1^{2*}=x_1+1$, we can write the following
\begin{dmath}
	c_e(x_1+1)+\sum_{a_0\in\cal{A}}\sum_{a_1\in\cal{A}}p_0(a_0)p_1(a_1)f(a_0+x_1+1-x_1-1,a_1) \leq c_e x_1+\sum_{a_0\in\cal{A}}\sum_{a_1\in\cal{A}}p_0(a_0)p_1(a_1)f(a_0+x_1+1-x_1, a_1),
\end{dmath}
and
\begin{dmath}
	c_e(x_1+1)+\sum_{a_0\in\cal{A}}\sum_{a_1\in\cal{A}}p_0(a_0)p_1(a_1)f(a_0,a_1) \leq c_e x_1+\sum_{a_0\in\cal{A}}\sum_{a_1\in\cal{A}}p_0(a_0)p_1(a_1)f(a_0+1, a_1).
\end{dmath}

Then, $y_1^{1*}=x_1$ since $f(x)$ is increasing and componentwise convex in $x_0$. As a result, we have 

\begin{equation}
T_{es}f(x)=c_ex_1+\sum_{a_0\in\cal{A}}\sum_{a_1\in\cal{A}}p_0(a_0)p_1(a_1) f(a_0, a_1)
\end{equation}

\begin{equation}
T_{es}f(x+e_1+e_2)=c_e(x_1+1)+\sum_{a_0\in\cal{A}}\sum_{a_1\in\cal{A}}p_0(a_0)p_1(a_1)f(a_0, a_1),
\end{equation}

and that $T_{es}f(x) \leq T_{es}f(x+e_1+e_2)$.

Now we will show that $T_{es}$ is componentwise convex. We will first prove componentwise convexity in $x_0$ which is expressed by the inequality below:

\begin{equation}
2T_{es}f(x+e_1)  \leq T_{es}f(x+2e_1) +T_{es}f(x) \label{Eq:comconvx0},
\end{equation}

where
\begin{equation}
 T_{es} f(x+e_1)=\min_{y_1\in Y_{x^{3}}}\{c_e y_1+\sum_{a_0\in\cal{A}}\sum_{a_1\in\cal{A}}p_0(a_0)p_1(a_1)f(a_0+x_1-y_1,a_1)\},
\end{equation}

\begin{equation}
 T_{es} f(x+2e_1)=\min_{y_1\in Y_{x^{4}}}\{c_e y_1+\sum_{a_0\in\cal{A}}\sum_{a_1\in\cal{A}}p_0(a_0)p_1(a_1)f(a_0+x_1-y_1,a_1)\},
\end{equation}

and $Y_{x^3}=\{0,1,\dots, \min\{x_1,(M-x_0-1)^+\}\}$, $Y_{x^4}=\{0,1,\dots, \min\{x_1, (M-x_0-2)^+\}\}$. Based on the service capacity, the following three cases can occur:

\begin{itemize}
\item[(i)] $0\leq x_1 \leq (M-x_0-2)^+$, and $Y_{x^4}=Y_{x^3}=Y_{x^1}$.
\item[(ii)] $x_1 = (M-x_0-1)^+$, and $Y_{x^4} \subset Y_{x^3}=Y_{x^1}$.
\item[(iii)] $ x_1 \geq(M-x_0)^+$ and $Y_{x^4}\subset Y_{x^3}\subset Y_{x^1}$.
\end{itemize}

Let the minimizer for $T_{es}f(x)$ be $y_1^{1*}$. Based on this, the following is true.

\begin{equation}
c_e \leq \sum_{a_0\in\cal{A}}\sum_{a_1\in\cal{A}}p_0(a_0)p_1(a_1)f(a_0+x_1-y_1^{1*}+1,a_1) - \sum_{a_0\in\cal{A}}\sum_{a_1\in\cal{A}}p_0(a_0)p_1(a_1)f(a_0+x_1-y_1^{1*},a_1)
\end{equation}

Since $f(x)$ is componentwise convex in $x_0$, if $y_1^{1*} \in Y_{x^3}$, then $y_1^{3*}=y_1^{1*}$; otherwise $y_1^{3*}=y_1^{1*}-1$ (notice that $ Y_{x^1}$ and $ Y_{x^3}$ decision sets can only differ by one element). Similarly, if $y_1^{3*} \in Y_{x^4}$, then $y_1^{4*}=y_1^{3*}$; otherwise $y_1^{4*}=y_1^{3*}-1$. So, we have either (a) $y_1^{1*}=y_1^{3*} =y_1^{4*}=y $, (b) $y_1^{1*}=y_1^{3*}=y_1^{4*}+1=y $, or (c) $y_1^{1*}=y_1^{3*}+1 =y_1^{4*}+2=y $.  

In case of (a), inequality (\ref{Eq:comconvx0})  holds at equality. In case of (b), by plugging the corresponding decisions in (\ref{Eq:comconvx0}), the inequality reduces to the following.

\begin{equation}
c_e \leq \sum_{a_0\in\cal{A}}\sum_{a_1\in\cal{A}}p_0(a_0)p_1(a_1)f(a_0+x_1-y+1,a_1) - \sum_{a_0\in\cal{A}}\sum_{a_1\in\cal{A}}p_0(a_0)p_1(a_1)f(a_0+x_1-y,a_1) \label{Eq:b}
\end{equation}

We know that inequality (\ref{Eq:b}) holds because $y_1^{1*}=y$ is the minimizer for $T_{es}(x)$ and provides a lower value than $y-1$. In case of (c), (\ref{Eq:comconvx0}) is equivalent to 

\begin{eqnarray}
2\sum_{a_0\in\cal{A}}\sum_{a_1\in\cal{A}}p_0(a_0)p_1(a_1)f(a_0+x_1-y+1,a_1) \leq \sum_{a_0\in\cal{A}}\sum_{a_1\in\cal{A}}p_0(a_0)p_1(a_1)f(a_0+x_1-y+2,a_1) \nonumber \\+\sum_{a_0\in\cal{A}}\sum_{a_1\in\cal{A}}p_0(a_0)p_1(a_1)f(a_0+x_1-y,a_1),
\end{eqnarray}

which holds due to componentwise convexity of $f(x)$ in $x_0$. 

%%%%%%%%%%%%%%%%%%%%%%%%%%%%%%%%%%
To prove the componentwise convexity of $T_{es}$ in $x_1$, we need to show the following inequality holds:

\begin{equation}
2T_{es}f(x+e_2)  \leq T_{es}f(x+2e_2) +T_{es}f(x) \label{Eq:comconvx1},
\end{equation}

where
\begin{equation}
 T_{es} f(x+e_2)=\min_{y_1\in Y_{x^{5}}}\{c_e y_1+\sum_{a_0\in\cal{A}}\sum_{a_1\in\cal{A}}p_0(a_0)p_1(a_1)f(a_0+x_1+1-y_1,a_1)\},
\end{equation}

\begin{equation}
 T_{es} f(x+2e_2)=\min_{y_1\in Y_{x^{6}}}\{c_e y_1+\sum_{a_0\in\cal{A}}\sum_{a_1\in\cal{A}}p_0(a_0)p_1(a_1)f(a_0+x_1+2-y_1,a_1)\},
\end{equation}

and $Y_{x^5}=\{0,1,\dots, \min\{x_1+1,(M-x_0)^+\}\}$, $Y_{x^6}=\{0,1,\dots, \min\{x_1+2, (M-x_0)^+\}\}$. Based on the service capacity, the following three cases can occur:

\begin{itemize}
\item[(i)] $ (M-x_0)^+ \leq x_1$, and $Y_{x^1}=Y_{x^5}=Y_{x^6}$.
\item[(ii)] $(M-x_0)^+ = x_1+1$, and $Y_{x^1} \subset Y_{x^5}=Y_{x^6}$.
\item[(iii)] $(M-x_0)^+ \geq  x_1+2$ and $Y_{x^1} \subset Y_{x^5}\subset Y_{x^6}$.
\end{itemize}

Let the minimizer for $T_{es}f(x)$ be $y_1^{1*}$. Based on this, the following is true.

\begin{equation}
c_e \leq \sum_{a_0\in\cal{A}}\sum_{a_1\in\cal{A}}p_0(a_0)p_1(a_1)f(a_0+x_1-y_1^{1*}+1,a_1) - \sum_{a_0\in\cal{A}}\sum_{a_1\in\cal{A}}p_0(a_0)p_1(a_1)f(a_0+x_1-y_1^{1*},a_1) \label{Eq:x1Y1}
\end{equation}

Since $f(x)$ is componentwise convex in $x_0$, the following is also true for any $\eta$.
\begin{equation}
\eta c_e \leq \sum_{a_0\in\cal{A}}\sum_{a_1\in\cal{A}}p_0(a_0)p_1(a_1)f(a_0+x_1-y_1^{1*}+\eta,a_1) - \sum_{a_0\in\cal{A}}\sum_{a_1\in\cal{A}}p_0(a_0)p_1(a_1)f(a_0+x_1-y_1^{1*},a_1) \label{Eq:x1Y1}
\end{equation}
Thus, $y_1^{5*}$ is either equal to $y_1^{1*}$ or $y_1^{1*}+1$. Due to the same reasoning, $y_1^{6*}$ is either equal to $y_1^{5*}$ or $y_1^{5*}+1$. That is, we have either (a) $y_6^{1*}=y_1^{5*} =y_1^{1*}=y $, (b) $y_1^{6*}=y_1^{5*}=y_1^{1*}+1=y $, or (c) $y_1^{6*}=y_1^{5*}+1 =y_1^{1*}+2=y $.

In case of (a), (\ref{Eq:comconvx1}) holds because $f(x)$ is componentwise convex in $x_0$. In case of (b),  (\ref{Eq:comconvx1}) holds because $y_1^{5*}=y$ and it provides a lower value than $y-1$. In case of (c), (\ref{Eq:comconvx1}) holds at equality. Therefore, $T_{es}$ is componentwise convex in $x_1$.

\end{proof}

\section{Computing the Threshold Policy }

\label{App:heuristic}
%%Time-first heuristic
\begin{algorithm}[H] 
\scalefont{0.8}
\caption{Single-Server Systems}\label{Alg:singleserver}
  \begin{algorithmic}[1]
  \State{current state: $\{x_{0}, x_{1}, \dots, x_{K-1}\}$}
  \State{$remaining\_capacity \gets M$} 
  %\State{$new\_state$ $\gets$ $current\_state$}
   \If {$x_{0} \geq remaining\_capacity$}  %C
   \State {$y_{0}$ $\gets$ $x_{0}$}
   \State {$y_{j}$ $\gets$ 0 for all $j\geq1$}
    \Else
    \State{$y_{0}\gets x_{0}$}
		\State{$remaining\_capacity$ $\gets$ $remaining\_capacity - x_{0}$}
    \For{$j:=0\textrm{ upto } =K-2$}  %B
      \State{$\theta \gets\frac{1+p_{j}(0)-p_{j}(0)p_{j}(1)-p_{j}^{2}(0)}{1-p_{j}^{2}(0)-p_{j}(0)p_{j}(1)}  $}
    %\If{($\theta c_{s} \leq c_{o}$ \& $new\_state[j+1] >0)$ $\|$ $(c_{s}\leq c_{o} \leq \theta c_{s}$ $\&$ $new\_state[j+1] >1$) }  %4
    \If{($\theta c_{e} \leq c_{o}$ \& $x_{j+1} >0)$}  %A
    \State{$y_{j+1}$ $\gets$ $\min\{remaining\_capacity,s_{j+1}  \}$}
    \Else if    { ($c_{e}$ $\leq$ $c_{o}$ $\leq$ $\theta c_{e}$ $\&$ $x_{j+1} >1$)}
    \State{$y_{j+1}\gets \min\{x_{j+1} -1,remaining\_capacity \} $}
    \EndIf  %A
    \State {$remaining\_capacity \gets remaining\_capacity -y_{j+1}$}   
    \EndFor   %B
\EndIf  %C
  \end{algorithmic}
\end{algorithm}

%multiple-server heuristic
\begin{algorithm}[H] 
\scalefont{0.8}
\caption{Multiple-Server Systems}\label{Alg:singleserver}
  \begin{algorithmic}[1]
  \State {Create the state space assuming $K=2$}
 \For{$j:=0\textrm{ upto } =K-2$} 
  \State{old\_average\_cost $\gets$ $\infty$}
  \For{$s:=A\textrm{ downto } =0$} 
  \State{current\_average\_cost $\gets$ average cost for $(0, s)$ threshold using $p_{j}$ probabilities}
  \If{current\_average\_cost $>$ old\_average\_cost}
  \State{$threshold_{j}$ $\gets$ $s+1$}
  \State{Break}
  \Else
  \State{old\_average\_cost $\gets$ current\_average\_cost}
  \EndIf
  \EndFor
  \EndFor
  \Comment{We have determined the local optimal thresholds}

  \State{current state: $\{x_{0}, x_{1}, \dots, x_{K-1}\}$}
  \State{$remaining\_capacity \gets M$} 
   \If {$x_{0} \geq remaining\_capacity$}  %C
   \State {$y_{0}$ $\gets$ $x_{0}$}
   \State {$y_{j}$ $\gets$ 0 for all $j\geq1$}
    \Else
    \State{$y_{0}\gets x_{0}$}
		\State{$remaining\_capacity$ $\gets$ $remaining\_capacity - x_{0}$}
    \For{$j:=0\textrm{ upto } =K-2$}  %B
    %\If{($\theta c_{s} \leq c_{o}$ \& $new\_state[j+1] >0)$ $\|$ $(c_{s}\leq c_{o} \leq \theta c_{s}$ $\&$ $new\_state[j+1] >1$) }  %4
    \If{$x_{j+1} >threshold_{j}$}  %A
    \State{$y_{j+1}$ $\gets$ $\min\{x_{j+1} -threshold_{j},remaining\_capacity \}$}
    \Else
    \State{$y_{j+1}\gets 0 $}
    \EndIf  %A
    \State {$remaining\_capacity \gets remaining\_capacity -y_{j+1}$}   
    \EndFor   %B
\EndIf  %C
  \end{algorithmic}
\end{algorithm}

\end{document}